\documentclass[12pt]{article}
\usepackage[T1]{fontenc}
\usepackage[utf8]{inputenc}
\usepackage[english]{babel}
\usepackage{authblk}

\usepackage{hyperref}
\usepackage{lmodern}
\usepackage{amsthm}
\usepackage{pdfpages}
\usepackage{listings}
\usepackage{multirow}
\usepackage{color, colortbl}
\usepackage{colortbl}
\usepackage{amsmath}
\usepackage{setspace}
\usepackage[noend]{algpseudocode}
\usepackage{enumerate}
\usepackage{mathtools,amssymb,amsthm}
\usepackage{color}
\usepackage{enumitem}
 \usepackage{mathtools}
 \usepackage{epigraph}
\usepackage{fancyhdr}
\usepackage{amsfonts}
\usepackage{mathrsfs}
\usepackage{tikz-cd}
\usepackage{fancybox}
\usepackage[all]{xy}
\usepackage{pifont}
\newtheorem{thm}{Theorem}[section]
\newtheorem{cor}{Corollary}[section]

\newtheorem{ex}{Example}[section]

\newtheorem{lemma}{Lemma}[section]
\newtheorem{defo}{Definition}[section]

\newtheorem{prop}{Proposition}[section]

\usepackage{blindtext}
\title{\textbf{ON GRADED NR-CLEAN RINGS}}
\author{ I. Namrok\footnote{namrokismail@gmail.com}}
\affil{}
\date{}

\begin{document}
\maketitle
\begin{abstract} This paper presents an extension of the concept of NR-clean introduced in \cite{1} to graded ring theory. We define and explore graded NR-clean rings, which generalize the  class of graded U-nil clean previously studied in \cite{66}. We provide examples and basic properties of such rings, and investigate their extensions to including matrix rings, group rings, as well as trivial ring extensions. 
\end{abstract}
\vskip0.1in
\textbf{keywords}: Graded rings and modules, NR-clean ring, group ring, matrix ring, trivial ring extension.
\section{Introduction}

In 2015, Khashan \cite{1} introduced a new interesting class of rings called \textit{NR-clean}. Indeed, Khashan has investigated some extensions of such rings, in particular matrix rings. Motivated by the importance of graded ring theory in many areas (Algebraic geometry, Representation theory, Quantum Mechanics,...). We aim in this paper  to extend NR-clean notion to graded rings. The idea of investigating some classes of ring from graded ring theory view has been inspiring lately many authors; for more details the reader is referred to \cite{12,19,16,17,15}.

Throughout this paper, all rings are associative with unity. Let $R$ be a ring; $J(R)$, $U(R)$, $N(R)$ and $Idem(R)$  denote respectively, the Jacobson radical of $R$, the multiplicative group of units of $R$, the set of nilpotent elements of $R$ and the set of idempotents of $R$.

Let $G$ be a group with identity $e$. We observe a $G$-graded ring, that is a ring $R$ such that $R=\oplus_{g\in G} R_g$ where
 $\{R_g\}_{g\in G}$ is a family of additive subgroups of $R$, and $R_gR_h\subseteq R_{gh}$ whenever $g,h\in G$. Elements of the set $h(R)= \bigcup_{g\in G}R_g $ are said \textit{homogeneous}; and for $g\in G$, the subgroup $R_g$ is called the $g$-th component of $R$, and elements of $R_g$ are said of \textit{degree} $g$. If $R=\oplus_{g\in G} R_g$ is $G$-graded, then every homogeneous idempotent element of $R$ must come from $R_e$. Moreover, if $f$ is a homogeneous idempotent element of $R$, then  $fR$ is a $G$-graded ring whose unity is $f$, and its $g$-th component is $(fR)_g=fR_g$ where $g\in G$.

 Given $n$ $G$-graded rings $A_1,R_2,\dots,A_n$. The direct product $R=\prod_{i=1}^nA_i$ is a $G$-graded ring, where its $g$-th component is $A_g=\prod_{i=1}^nA_i^g$, with $A_i^g$ is the $g$-th component of the ring $A_i$.
 
 Given two $G$-graded rings $R$ and $S$; a ring homomorphism $\varphi:R \longrightarrow S$ is said to be \textit{graded-homomorphism} if $\varphi(h(R)) \subseteq h(S)$; and if moreover $\varphi(R_g)\subseteq S_g$ for every $g\in G$, then $\varphi$ is called \textit{degree-preserving homomorphism}. 
 
A proper right (left or two-sided) ideal $I$ of a $G$-graded ring $R= \oplus_{g\in G}R_g$ is called \textit{homogeneous} (or \textit{graded}), if $I= \oplus_{g\in G}I\cap R_g$; equivalently, if the ideal $I$  can be generated by homogeneous elements of $R$.  In addition, a homogeneous ideal is called \textit{graded-nil}, if every homogeneous element of this ideal is nilpotent. 

If $I$ is a two-sided homogeneous ideal of $R$, then $R/I$ is a $G$-graded ring where its $g$-th component is $(R/I)_g=R_g/(I\cap R_g)$.

A homogeneous ideal $I$ of a graded ring $R$ is called \textit{graded-maximal}, if there is no other proper homogeneous ideal of $R$ which contains $I$. Moreover, the intersection of all graded-maximal right ideals of a graded ring $R$ defines the \textit{graded Jacobson radical} of $R$, denoted $J^g(R)$. Further, according to Proposition 2.9.1 in \cite{6} we have that $J^g(R)$ is a homogeneous two-sided ideal.

Let $R= \oplus_{g\in G}R_g$ is a $G$-graded ring. A $G$-graded $R$-module is an $R$-module $E$, such that $E= \oplus_{g\in G}E_g$, where $\{E_g\}_{g\in G}$ are additive subgroups of $E$, and $R_hE_g\subseteq E_{hg}$ for all $g,h\in G$.

Let $A$ be a ring and $E$ an $(A,A)$-bimodule. We denote by $A\propto E$ the set of pairs $(a,e)$ with pairwise addition and multiplication given by $(a,e)(b,f)=(ab,af+eb)$; the ring $R:=A\propto E$ is called the \textit{trivial extension} of $A$ by $E$. For more background of trivial ring extensions of commutative rings, we invite the reader to check \cite{20,23,21}.  If $A$ is a $G$-graded ring  and $E$ a $G$-graded $(A,A)$-bimodule, then according to \cite{20} $R=A\propto E$ is $G$-graded with the $g$-th component $R_g=A_g\oplus E_g$ for every $g\in G$.

Let $R= \oplus_{g\in G}R_g$ be a $G$-graded ring. According to \cite{8}, the group ring $R[G]$ is $G$-graded with the $g$-th component is $(R[G])_g= \oplus_{h\in G}R_{gh^{-1}}h$ and with the multiplication defined by $(r_gg')(r_hh')=r_gr_h(h^{-1}g'hh')$, where $g,g',h$, $h' \in G$,$r_g \in R_g$ and $r_h\in R_h$. If $H$ is a normal subgroup of $G$, then according to \cite{6} $R[H]$ can be viewed as a $G$-graded ring, where its $g$-th component is $(R[H])_g=\oplus_{h\in H}R_{gh^{-1}}h$. Moreover, the ring $R$ can be seen as a $G/H$-graded ring (see page 178 in \cite{6}) as follow: $R=\bigoplus_{C\in G/H}R_C$, where $R_C=\oplus_{x\in C}R_x$ for $C\in G/H$. Clearly $R[H]$ can also be observed as a $G/H$-graded ring.

Let $R=\oplus_{g\in G}R_g$ be a $G$-graded ring and $n$ a natural number. Indeed, for every $\sigma=(h_1,h_2,\dots,h_n) \in G^n$, the matrix ring $M_n(R)$ is $G$-graded where its $g$-th component is $M_n(R)_g(\sigma)=\begin{pmatrix}
R_{h_1gh_1^{-1}} &R_{h_1gh_2^{-1}}&\dots&R_{h_1gh_n^{-1}}\\R_{h_2gh_1^{-1}}&R_{h_2gh_2^{-1}}&\dots&R_{h_2gh_n^{-1}}\\ \vdots&\vdots&\dots&\vdots \\ R_{h_ngh_1^{-1}}&R_{h_ngh_2^{-1}}&\dots&R_{h_ngh_n^{-1}} \end{pmatrix}$ for every $g\in G$. Throughout this paper, if we consider the previous grading of the matrix ring $M_n(R)$, then we denote it $M_n(R)(\sigma)$.

An element $r$ of a ring $R$ is called \textit{regular} (in the sense of Von Neumann), if $r=rar$ for some $a\in R$. From graded ring theory view; if $R=\oplus_{g\in G}R_g$ is a $G$-graded ring; a homogeneous element $r\in R_g$ (where $g\in G$) is called \textit{graded-regular}, if $r=rar$ for some $a\in h(R)$. Indeed, since the sum $\sum_{g\in G}R_g$ is direct, then $a\in R_{g^{-1}}$. In this paper, we denote by $gr(R)$ the set of all graded-regular elements of $R$.

Khashan \cite{1} has defined a ring to be \textit{NR-clean}, if every element of $R$ is a sum of a regular element and a nilpotent. Concerning this paper, we define a graded ring to be \textit{graded NR-clean}, if every homogeneous element can be written as a sum of a homogeneous nilpotent and a graded-regular element. 

We begin by giving in the second section, examples and basic properties of graded NR-clean rings. We also investigate the relation between the two classes of graded U-nil clean rings and graded NR-clean rings.

Section 3 is devoted to the transfer of  the graded NR-clean property to some ring extensions. Indeed, we examine first in the main result of this paper, the question of when the matrix ring is graded NR-clean; further, we give  conditions under which the group ring is graded NR-clean. Also, we investigate trivial ring extensions of graded NR-clean rings.

Let $G$ be a group with identity $e$. Throughout this paper, unless otherwise specified, $G$ is  the grading group of graded rings.

\section{Graded NR-clean rings}

This section provides essential properties of graded NR-clean rings that are utilized throughout this paper. We also furnish an example of graded NR-clean ring. Further, we demonstrate that graded U-nil clean rings is a proper subclass of graded NR-clean rings. Moreover, we establish that having the NR-clean property of $e$-component of a graded ring, where $e$ is the identity of the grading group, does not necessarily imply that the entire ring is graded NR-clean.

\begin{defo}\label{d1}
A homogeneous element $x$ of a $G$-graded ring $R$ is called graded NR-clean if $x=r+n$ where $r\in gr(R)$ and $n\in N(R)\cap h(R)$. The ring $R$ is called graded NR-clean if every  element of $h(R)$ is graded NR-clean.
\end{defo}

Let us notice, that given a graded ring $R=\oplus_{g\in G} R_g$, and $a=r+n$   a graded NR-clean decomposition of  $a \in R_g$ (where $g\in G$) then $r,n\in R_g$. Indeed,  since the sum $\sum_{g\in G} R_g$ is direct, then  $r$ and $n$ have also the degree $g$ (i.e. $r,n\in R_g$). 


In the next example, we construct a $G$-graded ring where $G$ is a cyclic group of order three.

\begin{ex}\label{e1}
Let $A$ be an NR-clean ring (for instance, $A= \mathbb{Z}_4$ is an NR-clean ring according to \cite{1}). Suppose that $G=\{e,g,h\}$ is a cyclic group of order three. Then, the matrix ring $R=M_2(A)=R_e\oplus R_g \oplus R_h$ is $G$-graded where its components are: $R_e= \begin{pmatrix} A & 0 \\ 0 & A \end{pmatrix}$, $R_g=\begin{pmatrix} 0 & A \\ 0 & 0 \end{pmatrix}$ and $R_h=\begin{pmatrix} 0 & 0 \\ A & 0 \end{pmatrix}$. On other hand, according to Proposition 6 in \cite{1}, $A^2$ is $G$-graded NR-clean. Further, the two rings  $R_e$ and $A^2$ are isomorphic; hence $R_e$ is an NR-clean ring. In addition, every element from $R_g$ and $R_h$ is nilpotent. Thus, $R$ is $G$-graded NR-clean.
\end{ex}

According to, a homogeneous element $r$ of a graded ring $R$, is called \textit{graded unit regular} if $r=fu$ where $f$ is a homogeneous idempotent, and $u$ is a homogeneous unit. Clearly, every graded unit regular element is graded-regular. In \cite{66}, the author defined a graded ring  to be \textit{graded U-nil clean}, if every homogeneous element can be written as a sum of a graded unit regular and a homogeneous nilpotent. It is clear that graded U-nil clean rings is a subclass of graded NR-clean rings. Indeed, Example \ref{e2}  shows that this containment is proper.

\begin{ex}\label{e2}
    Let $S=K[[x]]$ where $K$ is a field such that $char(K)\neq 2$ and let $Q$ be the field of fractions of $S$. We define:
    $$ A=\{a\in End(S_K):\exists q\in Q~and~ \exists n \geq 1~ with~ a(b)=q(b)~\forall b\in x^nS\}.$$ According to \cite{2}, $A$ is an NR-clean ring which is not U-nil clean. Now, let $R=M_2(A)$ and $G$ be a cyclic group of order three. According to Example \ref{e1}, $R$ is $G$-graded NR-clean. Additionally, we have that $R_e \cong A\times A$ which is not U-nil clean (according to Proposition 2.5 in \cite{2} ). Hence, by Proposition 3.1 in, $R$ is not graded U-nil clean.
\end{ex}

\begin{prop}\label{p1}
Let $R=\oplus_{g\in G}R_g$ be a $G$-graded NR-clean ring. Then,
\begin{enumerate}
    \item $R_e$ is an NR-clean ring.
    \item Every graded-homomorphic image of $R$ is $G$-graded NR-clean.
    
\end{enumerate}
\end{prop}

\begin{proof}
\begin{enumerate}
    \item Let $a \in R_e$. Since $R$ is graded NR-clean, then $a=r+n$ where $r\in gr(R)\cap R_e$ and $n \in N(R)\cap R_e$. hence, $a$ is an NR-clean element of the ring $R_e$. Therefore, the ring $R_e$ is NR-clean.
    \item Let $S$ be a G-graded ring and $\varphi :R \longrightarrow S$  a graded-epimorphism. Assume that $R$ is graded NR-clean and let $a\in h(R)$. Since $R$ is graded NR-clean, then we can write $a=r+n$ where $ r \in gr(R)$ and $n \in N(R)\cap h(R)$. Hence,  $\varphi(a)=\varphi(r)+\varphi(n)$. Clearly $\varphi(n) \in N(S)\cap h(S)$. It remains to show that the homogeneous element $\varphi(r)$ is a graded-regular element of the ring $S$. Indeed, we know that there exists $x\in h(R)$ such that $r=rxr$; thus $\varphi(r)=\varphi(r)\varphi(x)\varphi(r)$. Since $\varphi$ is a graded-homomorphism, then $\varphi(x) \in h(S)$. Therefore $\varphi(r)\in gr(S)$. Hence, $S$ is graded NR-clean.
\end{enumerate}
\end{proof}

    Let us return to Proposition \ref{p1}. One may ask, if the converse of the first point holds true. In fact, this is not true in general as shows the next example.

\begin{ex}
    Let's consider the $\mathbb{Z}$-graded ring  $R:=\oplus_{n\in \mathbb{Z}}R_n$, where $R_n=\mathbb{Z}_2X^n$ for $n \geq 0$ and $R_n=0$ for $n < 0$. Indeed, we have $R_0=\mathbb{Z}_2$ is a NR-clean ring. Though, the homogeneous element $X \in R_1$ is not graded NR-clean. Consequently, $R$ is not $\mathbb{Z}$-graded NR-clean.
\end{ex}

\begin{prop}\label{p3}
    Let $R_1,R_2,\dots,R_k$ be $G$-graded rings. Then, $R=\prod_{i=1}^kR_i$ is a $G$-graded NR-clean ring if and only if $R_i$ is graded NR-clean for all $i\in \{1,2,\dots,k\}$. 
\end{prop}

\begin{proof}
    "$\Longrightarrow$" Assume that $R$ is graded NR-clean. It is clear that the projection map $\pi_i:R \longrightarrow R_i$ is a graded-epimorphism for every $i\in \{1,2,\dots,k\}$. Thus, according to Proposition \ref{p1} $R_i$ is graded NR-clean for each $i\in \{1,2,\dots,k\}$.

    "$\Longleftarrow$" Suppose that $R_i$ is graded NR-clean for each $i$. Let $(a_1,a_2,\dots,a_k)$ be a homogeneous element of $R$ of degree $g$ where $g\in G$. For each $i$ we can choose $r_i \in gr(R_i)$ and $n_i \in N(R_i) \cap h(R_i)$ such that $a_i=r_i+n_i$. Obviously we have $r=(r_1,r_2,\dots,r_k) \in gr(R)$ and $n=(n_1,n_2,\dots,n_k) \in N(R)\cap h(R)$; therefore $a=r+n$ is a graded NR-clean element of $R$. It follows $R$ is graded NR-clean.
\end{proof}

In \cite{12}, the authors defined a graded ring to be \textit{graded nil-good}, if every homogeneous element is either nilpotent or can be written as a sum of homogeneous unit and a homogeneous nilpotent. In fact, in \cite{66} it has been proved that the class of graded nil-good rings is a proper subclass of the class of U-nil clean rings. Therefore, according to Example \ref{e2}, the class of graded U-nil clean rings  lies properly between the two  classes of graded nil-good rings and graded NR-clean rings. The next proposition shows that these three classes coincide under some conditions.

\begin{prop}
    Let $R=\oplus_{g\in G}R_g$ be a graded ring with $gr(R)\subseteq Z(R)$ and $Idem(R_e)=\{0,1\}$. Then, $R$ is graded NR-clean if and only if $R$ is graded nil-good.
\end{prop}

\begin{proof}
    Suppose that $R$ is graded NR-clean and let $x\in R_g$ where $g\in G$. We can write $x=r+n$ for some $r\in gr(R)\cap R_g$ and $n \in N(R)\cap R_g$. We know that there exists $a \in R_{g^{-1}}$ such that $r=rar$. Since, $ra\in Idem(R_e)$ and $r\in Z(R)$ then $r=0$ or $r\in U(R)$. Thus, $x$ is either nilpotent or a sum of a homogeneous unit and a homogeneous nilpotent; which means $x$ is graded nil-good. Therefore, $R$ is graded nil-good.
\end{proof}


\begin{lemma}\label{l1}
Let $R=\oplus_{g\in G} R_g$ be a graded ring and $I$ a graded-nil ideal of $R$. Then, every graded-regular element modulo $I$ can be lifted to a graded-regular element of $R$.
\end{lemma}

\begin{proof}
    Let $x\in R_g$ be a graded-regular element modulo $I$, where $g\in G$. Hence, there exists $y\in R_{g^{-1}}$ such that $\overline{x}=\overline{x}\overline{y}\overline{x} $. Therefore, we have that $\overline{xy}=\overline{xy}\overline{xy}$ in the $G$-graded ring $R/I$; thus, $\overline{xy}$ is an idempotent element of the ring $R_e/I_e$ (where $I_e=I\cap R_e$). Since $I_e$ is a nil ideal of the ring $R_e$, then according to \cite{5} $I_e$  is a lifting ideal in the ring $R_e$. Now, if we apply Lemma 2.3 in \cite{4} to the idempotent element $xy \in R_e$, there exists an idempotent $f \in R_exy$ such that $1-f\in R_e(1-xy)$. Hence, we can write $f=rxy$ and $1-f=s(1-xy)$, where $r,s\in R_e$. Let $z:=frx \in R_g$; we have then $zy=frxy=(rxy)(rxy)=f^2=f$. Therefore, $zyz=fz=f^2rx=frx=z$; then $z$ is a graded-regular element of $R$. On the other hand, we have
    \begin{align*} 
  z-x &= z(1-yx)-(1-zy)x\\&= frx(1-xy)-(1-f)x\\ 
   &= frx(1-xy)-s(1-xy)x\\&=(fr-s)(x-xyx) \in I.
   \end{align*}
    Therefore, $\overline{x}=\overline{z}$ in the ring $R/I$; hence, $x$ is lifted to the graded-regular element $z$. 
     \end{proof}

\begin{thm}\label{t1}
Let $I$ be a graded-nil ideal of a $G$-graded ring $R$.Then, $R$ is graded NR-clean if and only if $R/I$ is graded NR-clean.
\end{thm}

\begin{proof}
    $(\Longrightarrow)$ Suppose that $R$ is graded NR-clean. Clearly the canonical epimorphism of rings $\pi:R \longrightarrow R/I$ is a graded-homomorphism. Hence, according to Proposition \ref{p1}, $R/I$ is graded NR-clean as a graded-homomorphic image of $R$.

    $(\Longleftarrow)$ Suppose that $R/I$ is graded NR-clean, and let $a\in R_g$ where $R_g$ is the $g$-th component of $R$ and $g\in G$. There exist then $\overline{r}\in gr(R/I) $ and $\overline{n} \in N(R/I)$, such that $\overline{a}=\overline{r}+\overline{n}$. Now, according to Lemma \ref{l1} we may assume that $r\in gr(R)$. In addition, we have $\overline{a-r}=\overline{n} \in N(R/I)$. Hence, $a-r$ is a homogeneous nilpotent element modulo the graded-nil ideal $I$; thus, $a-r\in N(R)\cap h(R)$. Consequently, $R$ is graded NR-clean.
\end{proof}

Let us notice that, the equivalence in Theorem \ref{t1} doesn't hold true if the ideal is not graded-nil. Indeed, a counterexample is given in the next example.

\begin{ex}\label{e4}
    Let $G=\{e,g\}$ be a cyclic group of order two, and $R=M_2(\mathbb{Z})$ the ring of $2\times 2$ matrices with integer coefficients. Clearly $R$ is $G$-graded where its components are: $R_e= \begin{pmatrix} \mathbb{Z} & 0 \\ 0 & \mathbb{Z} \end{pmatrix}$ and $R_g=\begin{pmatrix} 0 & \mathbb{Z} \\ \mathbb{Z} & 0 \end{pmatrix}$. Now, we have $R_e\cong \mathbb{Z}^2$ which is not an NR-clean ring (see Proposition 6 in \cite{1}). Thus, according to Proposition \ref{p1}, $R$ is not graded NR-clean. On the other hand, let's consider the homogeneous ideal $I=M_2(3\mathbb{Z})$ of $R$. We have $I$ is not graded-nil, and $R/I \cong \begin{pmatrix} \mathbb{Z}_3 & 0 \\ 0 & \mathbb{Z}_3 \end{pmatrix} \oplus \begin{pmatrix} 0 & \mathbb{Z}_3 \\ \mathbb{Z}_3 & 0 \end{pmatrix} $ which is a $G$-graded ring. Since, the field $\mathbb{Z}_3$ is an NR-clean ring, then according to Proposition 6 in \cite{1}, $\mathbb{Z}_3\times \mathbb{Z}_3$ is NR-clean; and therefore $(R/I)_e\cong \mathbb{Z}_3\times \mathbb{Z}_3 $ is an NR-clean ring. In addition, every element of $(R/I)_g$ is either nilpotent or a unit (in particular graded-regular). Thus, $R/I$ is graded NR-clean.
\end{ex}


We  recall that  the  \textit{support} of a $G$-graded ring $R=\oplus_{g\in G}R_g$, is the set $sup(R)=\{g\in G, R_g \neq 0\}$.

\begin{prop}
    Let $R=\oplus_{g\in G}R_g$ be a graded NR-clean ring of finite support. Then, $J^g(R)$ is graded-nil.
\end{prop}

\begin{proof}
    Let $a\in J^g(R)$ be a homogeneous element. Let us choose $r\in gr(R)$ and $n \in N(R)\cap h(R)$ such that $a=r+n$. According to Lemma 3.1 in \cite{66} we have $n\in J^g(R)$; therefore $r\in J^g(R)$. Now, since the support of $R$ is finite, then by Corollary 2.9.4 in \cite{6} we have $J^g(R) \subseteq J(R)$. On the other hand, there exists $x \in h(R)$ such that $r=rxr$; hence $r(1-xr)=0$. In addition, since $r\in J(R)$, then we have $1-xr \in U(R)$; thus $r=0$. Consequently, $a=n \in N(R)$, which proves that  $J^g(R)$ is graded-nil.
\end{proof}

\begin{cor}
Let $R$ be a graded ring of finite support. Then, $R$ is graded NR-clean if and only if $R/J^g(R)$ is graded NR-clean.
\end{cor}

\section{Extensions of graded NR-clean rings}
This section is dedicated to the investigation of some extensions of graded NR-clean rings. We  establish in the first subsection, the main result (Theorem \ref{t3}) of this section, concerning the matrix ring. Then, in the second subsection we will discuss the behaviour of graded NR-clean property in trivial ring extensions and group rings as well.

\subsection{Matrix rings}

\begin{thm}

Suppose that $G$ is an Abelian group, and let $R$ be a $G$-graded ring and $n$ a natural number. Then, $R$ is graded NR-clean if and only if $T_n(R)(\sigma)$ is graded NR-clean for every $\sigma \in G^n$.
\end{thm}

\begin{proof}
    "$\Longrightarrow$" Suppose that $R$ is graded NR-clean, and let $\sigma \in G^n$. Let's consider the homogeneous ideal $I$ of the graded ring $T_n(R)(\sigma)$, consisting of matrices with zeros in the main diagonal. It is well known that $I$ is a nilpotent ideal of $T_n(R)(\sigma)$; hence, in particular it is graded-nil. Now, we consider the epimorphism $\phi: T_n(R)(\sigma) \longrightarrow R^n$ defined by: $\phi((a_{ij})_{n\times n })=(a_{11},a_{22},\dots,a_{nn})$.Now, let $A_g\in T_n(R)_g(\sigma)$ where $g\in G$; since $G$ is an Abelian group, then the main diagonal elements of $A_g$ are all from $R_g$. Therefore, $\phi(T_n(R)_g(\sigma)) \subseteq (R^n)_g$ for all $g\in G$, where $(R^n)_g$ is the $g$-th component of the $G$-graded ring $R^n$. Thus,  $\phi$ is a degree-preserving homomorphism. On the other hand, we can easily verify that the kernel of $\phi$ is exactly the ideal $I$. Therefore, $T_n(R)(\sigma)/I$ is a graded-isomorphic image of $R^n$. According to  Proposition \ref{p1}, $R^n$ is graded NR-clean. Thus, according to Theorem \ref{t1} $T_n(R)(\sigma)$ is graded NR-clean.

    "$\Longleftarrow$" Assume that $T_n(R)(\sigma)$ is graded NR-clean for every $\sigma \in G^n$. Let $\sigma=(e,e,\dots,e)$, and let's consider the degree-preserving epimorphism $\epsilon:T_n(R)(\sigma) \longrightarrow R$  defined by: $\epsilon((a_{ij})_{n\times n })=a_{11}$. According to Lemma 4.1 in \cite{66} we have $\ker(\epsilon)$ is a homogeneous ideal of the $G$-graded ring $T_n(R)(\sigma)$. Therefore, $R$ and $T_n(R)(\sigma)/\ker(\epsilon)$ are graded-isomorphic. In addition, according Theorem \ref{t1},  $T_n(R)(\sigma)/\ker(\epsilon)$ is graded NR-clean; hence $R$ is graded NR-clean.
    \end{proof}

In order to establish the main result of this section, we need the two following lemmas which characterise graded U-nil clean rings using central homogeneous idempotents.

\begin{lemma}\label{l2}
    Let $R=\oplus_{g\in G} R_g$ be G-graded  ring and $f$ a central homogeneous idempotent of $R$. We have, $R$ is graded NR-clean If and only if $fR$ and $(1-f)R$ are $G$-graded NR-clean rings.
\end{lemma}

\begin{proof}
Suppose that $fR$ and $(1-f)R$ are graded NR-clean. Clearly $1-f \in Idem(R_e)\cap Z(R)$; hence, peirce decomposition of $R$ is: $R=fR\oplus (1-f)R \cong \begin{pmatrix} fR & 0 \\ 0 & (1-f)R \end{pmatrix}$. Indeed, the ring $S=\begin{pmatrix} fR & 0 \\ 0 & (1-f)R \end{pmatrix}$ whose identity is $\begin{pmatrix} f & 0 \\ 0 & (1-f) \end{pmatrix}$, is $G$-graded with the $g$-component $S_g=\begin{pmatrix} fR_g & 0 \\ 0 & (1-f)R_g \end{pmatrix}$ for $g\in G$. Moreover, the rings $R$ and $S$ are graded-isomorphic via the graded-isomorphism $\varphi:S\longrightarrow R$ defined by $\varphi(\begin{pmatrix} fx & 0 \\ 0 & (1-f)y \end{pmatrix})=fx+(1-f)y$ where $x,y\in R$. Now, let $A=\begin{pmatrix} fx & 0 \\ 0 & (1-f)y \end{pmatrix}\in S_g$, where $x,y\in R_g$ and $g\in G$. According to the assumption, there exist $fr_1\in gr(fR)\cap R_g$, and $fn_1 \in N(R)\cap R_g$ such that $fx=fr_1+fn_1$. Moreover, we can choose $(1-f)r_2\in gr((1-f)R)$ and $(1-f)n_2 \in N(R)\cap R_g$ such that $(1-f)y=(1-f)r_2+(1-f)n_2$. Then,
$$A= \begin{pmatrix} fr_1 & 0 \\ 0 & (1-f)r_2 \end{pmatrix}+\begin{pmatrix} fn_1 & 0 \\ 0 & (1-f)n_2 \end{pmatrix}.$$
The previous  decomposition of $A$, is obviously a graded NR-clean decomposition. Thus, $S$ is a graded NR-clean ring, and consequently $R$ is graded NR-clean as well.

Conversely, suppose that $R$ is graded NR-clean. Let $fa\in fR_g$ be a homogeneous element of $fR$ of degree $g\in G$; according to the assumption there exist $r\in gr(R)$ and $n\in N(R)\cap h(R)$ such that $a=r+n$; therefore, $fa=fr+fn$. Now, since $f\in Z(R)$ then $fr\in gr(fR)$ and $fn\in N(fR)$; consequently $fR$ is graded NR-clean. Since $1-f\in Z(R)\cap R_e$ then, similarly $(1-f)R$ is also a graded NR-clean ring

\end{proof}

Using again the general version of peirce decomposition, We obtain the following lemma, whose proof is similar to the one of Lemma \ref{l2}. First, let us recall the definition of orthogonal idempotents. A set of idempotents $\{f_i\}$ is said to be orthogonal if $f_if_j = 0$ for all $i\neq j$.  

\begin{lemma}\label{l4}

Let $f_1,\dots,f_n$ be  orthogonal homogeneous idempotents of a $G$-graded ring $R$ such that, $1=f_1+\dots+f_n$ and $f_i\in Z(R)$ for each $i$.
Then, $R$ is graded NR-clean if and only if
 every $f_iR$ is a graded NR-clean ring. 
\end{lemma}

\begin{thm}\label{t3}
    
Let $R$ be a $G$-graded ring. If $R$ is graded NR-clean, then $M_n(R)(\sigma)$ is graded NR-clean for any natural number $n$ and every $\sigma \in G^n$.
\end{thm}

\begin{proof}
Let $n$ be a natural number and $\sigma \in G^n$. We denote by $E_{ii}$ the matrix having 1 on the $(i,i)$-position and 0 elsewhere. It is easy to see that $E_{11},\dots,E_{nn} $ are orthogonal idempotents of $M_n(R)_e(\sigma)$. Additionally, $E_{ii} \in Z(M_n(R)(\sigma))$ for each $i\in \{1,\dots,n\}$, and $E_{11}+\dots+E_{nn}=1$. Now, we consider  the ring homomorphism $\epsilon_i:E_{ii}M_n(R)(\sigma) \longrightarrow R$ given by $\epsilon_i(E_{ii}(a_{kj})_{n\times n})=a_{ii}$. Clearly, $\epsilon_i$ is a graded-isomorphism for every $i\in \{1,\dots,n\}$. Consequently, according to Proposition \ref{p1} $E_{ii}M_n(R)(\sigma)$ is graded NR-clean for every $i\in \{1,\dots,n\}$. Therefore, according to Lemma \ref{l4} $M_n(R)(\sigma)$ is $G$-graded NR-clean.
\end{proof}

\subsection{Group rings and trivial ring extensions}
\begin{thm}
Let $R$ be a $G$-graded ring, and $H$ a normal subgroup of $G$. Suppose that $G$ is locally finite $p$-group where $p$ is nilpotent in $R$. If $R$ is $G/H$-graded NR-clean, then $R[H]$ is $G/H$-graded NR-clean.
\end{thm}
\begin{proof}
    Assume that $R$ is $G/H$-graded NR-clean. According to the demonstration  of Theorem 2.3 in \cite{9}, we can suppose that $H$ is a finite $p$-group; therefore $H=\{h_1,\dots,h_n\}$ where $h_i\in H$ for  $i\in \{1,\dots, n\}$.  Let $\phi$ be the epimorphism $\phi: R[H] \longrightarrow R$ defined by $\phi(\sum_{i=1}^n r_ih_i)=\sum_{i=1}^n r_i$, where $r_i \in R$ for each $i\in \{1,\dots, n\}$. According to page 180 of \cite{6}, we have $\phi$ is a degree-preserving epimorphism of $G/H$-graded-rings. Thus, according Lemma 4.1 in \cite{66} $\ker(\phi)$ is a homogeneous ideal of the $G/H$-graded ring $R[H]$. Consequently, $R[H]/\ker(\phi)$ and $R$ are $G/H$-graded-isomorphic rings. Hence, $R[H]/\ker(\phi)$ is $G/H$-graded NR-clean. Additionally, since $p \in N(R)$ then by  Theorem 9 in \cite{10}, we have $\ker(\phi)$ is a nilpotent ideal of $R[H]$.  It follows that $\ker(\phi)$ is a graded-nil ideal. According to  Theorem \ref{t1}, $R[H]$ is $G/H$-graded NR-clean.
\end{proof}

\begin{prop}
    Let  $R$ be a $G$-graded ring. If $(fR)[G]$ and $((1-f)R)[G]$ are $G$-graded NR-clean, where  $f$ is a central homogeneous  idempotent of $R$, then $R[G]$ is graded NR-clean.
\end{prop}

\begin{proof}
    According to Lemma \ref{l2}, It is sufficient to prove that $f(R[G])$ and $(1-f)(R[G])$ are graded NR-clean. Indeed, the homomorphism of rings $\varphi:(fR)[G] \longrightarrow f(R[G])$ defined by $\varphi(\sum_{g\in G} fr_gg)=f\sum_{g\in G}r_gg $, is a $G$-graded-isomorphism. Whereas, $(fR)[G]$ is graded NR-clean, then $f(R[G])$ is graded NR-clean. Similarly, $(1-f)(R[G])$ is graded NR-clean. Therefore, according to Lemma \ref{l2}, $R[G]$ is graded NR-clean.
\end{proof}

\begin{prop}
    Let $R$ be a $G$-graded ring where $G=\{e,g\}$ is a finite group of order two. If $R$ is graded NR-clean, then $R[G]$ is graded NR-clean.
\end{prop}

\begin{proof}
Let $\sigma=(e,g)$. We have that $M_2(R)_e(\sigma)=\begin{pmatrix}
        R_e & R_g \\
        R_g &  R_e
\end{pmatrix}$ and $M_2(R)_g(\sigma)=\begin{pmatrix}
        R_g & R_e \\
        R_e &  R_g
\end{pmatrix}$. Moreover, We have that $(R[G])_e=R_ee+R_gg$ and $(R[G])_g=R_ge+R_eg$. Now,  let's consider the mapping $\alpha:R[G] \longrightarrow M_2(R)(\sigma)$, defined by $\alpha\big((a_ee+a_gg)+(b_ge+b_eg)\big)=\begin{pmatrix}
        a_e & a_g \\
        a_g &  a_e
    \end{pmatrix}+\begin{pmatrix}
        b_g & b_e \\
        b_e &  b_g
    \end{pmatrix}$, where $a_e,b_e\in R_e$ and $a_g,b_g \in R_g$. In fact, the mapping $\alpha$ is a graded-isomorphism of rings. Yet, according to Theorem \ref{t3} $M_2(R)(\sigma)$ is $G$-graded NR-clean. Thus, according to Proposition \ref{p1} $R[G]$ is graded NR-clean.
\end{proof}

\begin{prop}
Let $R$ be a graded ring. Then, if $R[G]$ is graded NR-clean then $R$ is an NR-clean ring.
\end{prop}

\begin{proof}
If $R[G]$ is graded NR-clean, then according to Proposition \ref{p1}, we have $(R[G])_e$ is NR-clean. According to Proposition 2.1 in \cite{8}, $(R[G])_e$ and $R$ are isomorphic. Therefore, by applying Proposition 3 in \cite{1}, we get $R$ is NR-clean.
\end{proof}

\begin{thm}
    Let $A$ be a $G$-graded ring and $M$ a $G$-graded $(A,A)$-bimodule. Then, $R=A\propto M$  is graded NR-clean if and only if $A$ is graded NR-clean.
\end{thm}

\begin{proof}
   Since $(0\propto M)^2=0$, then $0\propto M$ is a graded-nil ideal of $R$. Moreover, the homomorphism $\epsilon:R \longrightarrow A$ defined by $\epsilon((a,e))=a$, is a graded-epimorphism, whose  kernel is exactly $0 \propto M$; therefore $R/(0\propto M)$ and $A$ are graded-isomorphic. Now, according to Theorem \ref{t1}, $R$ is graded NR-clean if and only if $A$ is graded NR-clean.
\end{proof}


\begin{thebibliography}{9}
\addcontentsline{toc}{chapter}{Bibliographie}



\bibitem{3} Anderson, F. W., Fuller, K. R., \textit{Rings and Categories of Modules}, Berlin/Heidelberg/New York: Springer (1992).

\bibitem{20} Anderson, D.D., Winders, M., \textit{Idealization of a module}, J. Commut. Algebra \textbf{1}(1) (2009), 3-56.


\bibitem{12} Choulli, H., Mouanis, H., Namrok, I., \textit{Group graded rings with the nil-good property}, Communications in Algebra \textbf{50(4)} (2022), 1-10

\bibitem{19} Choulli, H., Kchit, O., Mouanis, H. and Namrok, I., \textit{On graded trinil clean rings}, ANNALI DELL’UNIVERSITA’ DI FERRARA (2022) https://doi.org/10.1007/s11565-022-00450-5

\bibitem{10} Connell, I. G., \textit{On the group ring}, Can. j. math. \textbf{15} (1963), 650–685.


\bibitem{23}  Glaz, S., \textit{Commutative coherent rings}, Lecture Notes in Mathematics,
\textbf{1371}, Springer-Verlag, Berlin (1989).

\bibitem{4} Gupta, R. N., Khurana, D., \textit{Lifting idempotents and projective covers}, Kyungpook Math. J. \textbf{41} (2001), 217-227.

\bibitem{21}  Huckaba, J.A., \textit{Commutative rings with zero divisors}, Marcel Dekker, New York, (1988).



\bibitem{16} Ilić-Georgijević, E.,  \textit{On graded 2-nil-good rings}, Kragujevac Journal of Mathematics \textbf{43}(4) (2019), 513-522.

\bibitem{17} Ilić-Georgijević, E.,  \textit{On graded $UJ$-rings}, Kragujevac Journal of Mathematics \textbf{43}(4) (2020), 513-522.


\bibitem{15} Ilić-Georgijević, E. and
 Şahinkaya, S.,  \textit{On graded nil clean rings}, Comm. Algebra \textbf{46}(9) (2018),
4079–4089.


\bibitem{1} Khashan, H., \textit{NR-clean rings}, Vietnam J. Math \textbf{44} (2016), 749-759.

\bibitem{2} Khashan, H., \textit{Rings whose elements are a sum of a unit regular and nilpotent}, Italian Journal of Pure and Applied Mathematics \textbf{46} (2021), 645–658.

\bibitem{5} Khurana, D., Lam, T. Y., \textit{Rings with internal cancellation}, Journal of Algebra \textbf{284}(1) (2004), 203–235.

\bibitem{66} Namrok, I., \textit{On graded U-nil clean rings}, arXiv Preprint, arXiv:2311.09331.


\bibitem{6} Năstăsescu, C., Van Oystaeyen, F. \textit{Methods of graded rings}, Lecture Notes in Mathematics. Berlin, Heidelberg: Springer, p. 1836.

\bibitem{8} Năstăsescu, C., \textit{Group rings of graded rings applications}, J. Pure Appl. Algebra \textbf{33(3)} (1984), 313–335.

\bibitem{9} Sahinkaya, S., Tang, G., Zhou, Y., \textit{Nil-clean group rings}, J. Algebra Appl. \textbf{16(07)} (2017), 1750135.

\end{thebibliography}
\end{document}